\documentclass[12pt]{article}

\usepackage[inner=32mm,outer=31mm,tmargin=30mm,bmargin=40mm]{geometry}
\usepackage[ngerman,english]{babel}

\usepackage{latexsym,amsfonts,amsmath,amssymb,epsfig,tabularx,amsthm,dsfont,mathrsfs}

\usepackage{graphicx}
\usepackage{enumerate}

\usepackage{booktabs}
\usepackage{titling}
\newcommand{\subtitle}[1]{%
  \posttitle{%
    \par\end{center}
    \begin{center}\large#1\end{center}
    \vskip0.5em}%
}

\usepackage{hyperref} 
\hypersetup{colorlinks=true,
        linkcolor=black,
        citecolor=black,
        urlcolor=blue}

\usepackage{pgfplots}

\usepackage{framed}
\usepackage{amscd}

\usepackage{tikz-cd}
\usepackage{tikz}
\usetikzlibrary{calc,intersections,patterns}

\usepackage{caption}
\theoremstyle{plain}

\newtheorem{theorem}{Theorem}[section]
\newtheorem{lemma}[theorem]{Lemma}

\newtheorem{corollary}[theorem]{Corollary}

\theoremstyle{definition}
\newtheorem{remark}[theorem]{Remark}

\renewcommand{\P}{{\mathbb P}}
\newcommand{\expect}{\operatorname{\mathbb{E}}}

\DeclareMathOperator{\Uniform}{unif}
\DeclareMathOperator{\Normal}{\mathcal{N}}

\newcommand{\ind}{\mathds{1}}

\DeclareMathOperator{\id}{id}

\DeclareMathOperator{\card}{card}
\newcommand{\ran}{\textup{ran}}

\newcommand{\nonada}{\textup{nonada}}

\newcommand{\eps}{\varepsilon}
\newcommand{\embed}{\hookrightarrow}

\renewcommand{\vec}{\boldsymbol}

\newcommand{\R}{{\mathbb R}}

\newcommand{\N}{{\mathbb N}}

\DeclareMathAlphabet{\mathup}{OT1}{\familydefault}{m}{n}

\newcommand{\widebar}[1]{\mbox{\kern1.5pt\hbox{\vbox{\hrule height 0.6pt \kern0.35ex
        \hbox{\kern-0.15em \ensuremath{#1 }\kern0.0em}}}}\kern-0.1pt}

\newlength{\fixboxwidth}
\usepackage{soul}

\usepackage{latexsym,amsfonts,amsmath,amssymb,mathrsfs}
\usepackage{url}
\usepackage{graphicx}
\usepackage{color}
\usepackage{euscript}
\usepackage{verbatim}
\usepackage{hyperref}

\definecolor{darkgreen}{rgb}{0,0.5,0}

\begin{document}

\title{Sharp bounds for non-adaptive randomized approximation of high-dimensional vectors}

\author{Robert J. Kunsch\thanks{RWTH Aachen University,
Mathematics for Uncertainty Quantification,
Pontdriesch~10--16,
52062 Aachen, Email: kunsch@mathc.rwth-aachen.de},
Marcin Wnuk\thanks{Fakultät für Informatik und Mathematik, 
Passau University, Instraße 33, 94032 Passau, 
Email: marcin.wnuk@uni-passau.de}}

\date{\today}

\maketitle
\begin{abstract}
    We study the complexity of approximating the finite-dimensional
    vector space embedding $\ell_p^m \embed \ell_q^m$ for $2 \leq p < q \leq \infty$
    based on non-adaptive randomized algorithms
    that use up to $n$ arbitrary linear functionals
    as information on a problem instance $\vec{x} \in \R^m$, where $n \ll m$.
    We prove lower bounds on the non-adaptive randomized approximation error
    with a joint dependence on~$(n,m)$
    matching previously known upper bounds.
\end{abstract}


\section{Introduction}
\label{sec:Intro}

We investigate the approximation of the embedding $\ell_p^m \hookrightarrow \ell_q^m$ 
for $2 \leq p < q \leq \infty$ by means of non-adaptive randomized algorithms,
where $\ell_p^m$ is $\R^m$ equipped with the $\ell_p$-norm.
The article is a follow-up to the papers~\cite{KNW24, KW24b, KW24c}, yet self-contained.
The problem of approximating the embeddings of finite-dimensional spaces is of interest,
mainly due to two factors: 
on the one hand, approximating those embeddings is a building block 
for approximating more complicated function space embeddings; see, e.g.~\cite{H92,Ma91}. 
On the other hand, finite-dimensional embeddings constitute a `minimal working example' 
where one can observe interesting theoretical phenomena, 
such as the fact that randomized adaptive algorithms can be much better 
than randomized non-adaptive algorithms~\cite{KNW24, KW24b, KW24c}.

We start by describing admissible algorithms.
A \emph{deterministic non-adaptive algorithm} for the problem $\ell_p^m \to \ell_q^m$ 
is any mapping of the form
\begin{equation*}
    A = \varphi \circ N,
\end{equation*}
where
\begin{itemize}
    \item $N: \ell^m_p \rightarrow \mathbb{R}^n$ is a linear mapping called \emph{information mapping},
        and
    \item $\varphi: \mathbb{R}^n \rightarrow \ell_q^m$, 
        called the \emph{reconstruction mapping}, 
        can be any mapping.
\end{itemize}
Thus, $A(\vec{x})$ first applies $n$ linear functionals to $\vec{x}$ (it `collects information')
and then transforms this information in some way to obtain the output.
One calls~$n$ the \emph{cardinality (or information cost)} of $A$ and writes
\begin{equation*}
    \card A = n.
\end{equation*}
A \emph{randomized non-adaptive algorithm} is a collection $A = (A^{\omega})_{\omega \in \Omega}$ 
of deterministic non-adaptive algorithms indexed by the elements of the sample space 
of some probability space $(\Omega, \Sigma, \mathbb{P})$,
where the mapping $\Omega \ni \omega \mapsto \card A^{\omega}$ is a random variable
and, for a fixed $\vec{x} \in \ell^m_p$,
the mapping $\Omega \ni \omega \mapsto A^{\omega}(\vec{x})$
is also a random variable, usually written as $A(\vec{x})$.
For technical reasons, we also assume that the mapping $(\omega,\vec{x}) \mapsto A^\omega(\vec{x})$
is measurable with respect to the Borel-$\sigma$-algebra on $\R^m$ 
and the product $\sigma$-algebra on $\Omega \times \R^m$.
We regard deterministic non-adaptive algorithms as a special instance of randomized algorithms,
in particular, we assume Borel-measurability.
Since in the rest of the paper we deal mainly with non-adaptive algorithms,
we will often simply speak of algorithms.

Given a randomized algorithm $A$, we define its \emph{error}
when approximating the embedding $\ell^m_p \hookrightarrow \ell^m_q$ as
\begin{equation*}
    e(A, \ell^m_p \hookrightarrow \ell^m_q)
        := \sup_{\vec{x}\colon \| \vec{x} \|_p \leq 1} \expect \| \vec{x} - A(\vec{x}) \|_q
\end{equation*}
where $\|\cdot\|_p$ is the classical $\ell_p$-norm.
The cardinality of a randomized algorithm is given by the worst case cost,
\begin{equation*}
    \card A := \sup_{\omega \in \Omega} \card A^{\omega}.
\end{equation*}
We are interested in the \emph{$n$-th minimal error} 
\begin{equation*}
    e^{\ran, \nonada}(n, \ell^m_p \hookrightarrow \ell^m_q) 
        := \inf_{A_n} e(A_n, \ell^m_p \hookrightarrow \ell^m_q),
\end{equation*}
where the infimum is taken over all randomized algorithms $A_n$ with cardinality at most $n$.
For $\eps > 0$, we also consider the dual notion of the \emph{$\eps$-complexity},
\begin{equation*}
    n^{\ran,\nonada}(\eps, \ell^m_p \hookrightarrow \ell^m_q) 
        := \inf\{ \card A \colon  e(A, \ell^m_p \hookrightarrow \ell^m_q) \leq \eps\},
\end{equation*}
where the infimum is taken over randomized non-adaptive algorithms.

In this article,
we focus on the regime ``$n \ll m$'' where the dimension of the problem $m$
is much larger than the cardinality $n$ of algorithms, namely, for \mbox{$n \leq m^\alpha$} with some constant $\alpha \in (0,1)$ and sufficiently large $m$.
The main result states that, in the case \mbox{$2 \leq p < q \leq \infty$}, one has
\begin{equation}\label{eq:MainResult}
    e^{\ran, \nonada}(n, \ell^m_p \hookrightarrow \ell^m_q)
        \asymp \min \left\{  1,\,
                        \left( \frac{m^{1-\frac{2}{p}} \cdot \log m}{n} \right)^{\frac{1}{2}\left(1-\frac{p}{q}\right)} 
                    \right\},  
\end{equation}
see Theorem~\ref{thm:Main}.
The upper error bounds are already known from~\cite[Thm~4.4]{KW24c}
and were obtained by means of denoising a linear randomized algorithm.
The contribution of this paper is a proof of the matching lower bounds.
For this we use Bakhvalov's trick,
where we switch from the randomized setting
to an average case setting for deterministic algorithms.
The average case setting is studied in Section~\ref{sec:avganalysis}.
The results of Section~\ref{sec:avganalysis}
are used to prove a randomized lower bound in Section~\ref{sec:rannonadaLB}.

Upper bounds are known for the non-adaptive randomized approximation of $\ell_p^m \embed \ell_q^m$ 
for the general parameter range~$1 \leq p < q \leq \infty$, see~\cite{KW24c}.
We conjecture that these algorithms are asymptotically optimal also for $1 \leq p < 2$.
The paper~\cite{KNW24} contains a lower bound for non-adaptive randomized approximation
which holds for the whole parameter range of $p$ and $q$, in particular $\ell_1^m \embed \ell_\infty^m$,
but it only shows that the error will be larger than a constant $\eps_0 = \frac{1}{100}$
if $n$ is too small compared to $m$.
This paper now proves lower bounds with joint dependence on $m$ and~$n$,
but only for $p \geq 2$.
For $p < 2$ the suitable average case setting appears to be more difficult to analyse.
We note that from the information-based complexity point of view,
the approximation problem is much harder for $p > 2$ because the error rates and complexity
exhibit a factor that is polynomially large in~$m$
while for \mbox{$p \in [1,2]$} the dependence on $m$ is only logarithmic:
\begin{equation} \label{eq:nonada_p<2}
    e^{\ran, \nonada}(n, \ell^m_p \hookrightarrow \ell^m_q)
        \preceq \min \left\{  1,\,
                        \left( \frac{\log m}{n} \right)^{\frac{1}{p}- \frac{1}{q}} 
                    \right\}, \quad 1 \leq p \leq 2, \; p < q \leq \infty.
\end{equation}
Upper bounds are also known for the adaptive randomized setting
where information functionals can be chosen depending on previous measurements,
see~\cite{KW24b,KW24c}.
Compared to the non-adaptive upper bounds in \eqref{eq:MainResult} and \eqref{eq:nonada_p<2},
adaptive upper bounds exhibit a $\log \log m$-dependence instead of a $\log m$-dependence;
still, for $p > 2$ the polynomial dependence on $m$~remains.
Proving lower bounds for the adaptive setting is yet another challenge, especially for $p < 2$.

Let us provide some intuition on the difficulties of approximating vectors from the unit ball of~$\ell_p^m$ with $p > 2$:
Randomized algorithms for $\ell_p^m \embed \ell_q^m$
aim to recover the largest entries (in absolute value) of~$\vec{x}$.
In randomized measurements,
the smaller entries of~$\vec{x}$ induce noise that scales with the $\ell_2$-norm,
which can be large for $p > 2$,
namely: If~\mbox{$\|\vec{x}\|_p \leq 1$} for $\vec{x} \in \R^m$,
then $\|\vec{x}\|_2 \leq m^{\frac{1}{2} - \frac{1}{p}}$ is a sharp estimate.
This gives an intuitive explanation for the factor $m^{1 - 2/p}$
that occurs in the error bounds.
The lower bound proof shows that this factor is indeed unavoidable,
at least if we restrict ourselves to non-adaptive algorithms.

\subsection*{Asymptotic notation}

We use asymptotic notation to compare functions $f$ and $g$ 
that depend on variables $(\eps,m)$ or $(n,m)$, 
writing \mbox{$f \preceq g$} if there exists a constant \mbox{$C > 0$} such that $f \leq C g$ holds.
\emph{Weak asymptotic equivalence} $f \asymp g$ means $f \preceq g \preceq f$.
We often state such asymptotic relations with restrictions on the variables,
e.g., for sufficiently large $m \geq m_0$ and $n \leq m^{\alpha}$ or $\eps \geq m^{-\beta}$, respectively,
meaning that the inequality $f \leq Cg$ only holds for pairs $(\eps,m)$ or $(n,m)$
that satisfy these restrictions.
The implicit constant $C$ may depend on other parameters such as $p$ and $q$.

\section{Average case analysis}
\label{sec:avganalysis}

We use Bakhvalov's technique (also known as Yao's principle)
to prove lower bounds for randomized methods
by switching to an average case setting.
For a (sub-)probability measure $\mu$ on $\ell_p^m$,
the $\mu$-average error for a deterministic algorithm \mbox{$A\colon \ell_p^m \to \ell_q^m$}
is defined as follows:
\begin{equation} \label{eq:emu}
  e^\mu(A,\ell_p^m \embed \ell_q^m)
    := \int \|A(\vec{x}) - \vec{x}\|_q \,d\mu(\vec{x}).
\end{equation}
Given a (sub-)probability measure~$\mu$ with support in the $\ell_p$-unit ball of $\R^m$,
Bakhvalov's technique states
\begin{equation} \label{eq:Bakhvalov}
    e^{\ran,\nonada}(n,\ell_p^m \embed \ell_q^m)
        \geq \inf_{A_n} e^{\mu}(A_n, \ell_p^m \embed \ell_q^m),
\end{equation}
where the infimum for the $\mu$-average error
is taken over non-adaptive deterministic algorithms~$A_n$
that use at most $n$~pieces of linear information.
This result is folklore, a proof can be found, e.g., in~\cite[Prop.~1.2]{Ku17}.
For the sake of simplicity,
in this paper we only
consider the error for randomized algorithms with a strict cardinality bound~$n$.
The setting with expected cardinality bound can be treated in a similar way, see~\cite[Sec~1.3]{Ku17}.

\subsection{Truncated Gaussian mixture}
\label{sec:mu}

The notation introduced in this subsection will be used throughout Section~\ref{sec:avganalysis}.
We start by defining the measure $\mu$ on the $\ell_p$-unit ball in~$\R^m$. 
In addition to $p \in [2,\infty)$, this measure depends on a parameter $k \in \N$ with $k \ll m$.
Let
\begin{equation} \label{eq:W_k}
  W = W_k^m := \{ \vec{w} \in \{-1,0,1\}^m \colon \vec{w}\ \text{with exactly $k$ non-zero entries}\}
\end{equation}
with $M := \# W = \binom{m}{k} 2^k$
and let $W = \{\vec{w}_1,\ldots,\vec{w}_M\}$ be an arbitrary enumeration of the elements of $W$;
we will write $[M] = \{1,\ldots,M\}$ for the index set.
If two elements $\vec{w}_i, \vec{w}_j \in W$ differ in at least $k$ entries,
then they have an $\ell_p$-distance bounded by \mbox{$\|\vec{w}_i - \vec{w}_j\|_p \geq k^{1/p}$}.
Rescaling to $\vec{u}_i := \frac{1}{2} \cdot k^{-1/p} \cdot \vec{w}_i$, 
we then have $\|\vec{u}_i\|_p = \frac{1}{2}$ and 
$\|\vec{u}_i - \vec{u}_j\|_q \geq \frac{1}{2} \cdot k^{-\left(\frac{1}{p} - \frac{1}{q}\right)}$.
Consider the random vector
\begin{equation} \label{eq:Xdef}
  \vec{X} = \vec{u}_I + \sigma \vec{Z}_m\,,
\end{equation}
where $I \sim \Uniform[M]$ and $\vec{Z}_m \sim \Normal(\vec{0},\id_m)$ are independent random variables
and $\sigma = \frac{1}{4\sqrt{p}}\, m^{-1/p} > 0$ is a scaling factor.
The probability measure $\widetilde{\mu} := \P^{\vec{X}}$
is not supported on the $\ell_p$-unit ball of $\R^m$,
so we introduce a truncation by the event
\begin{equation*}
  \Omega' 
    := \left\{\omega \in \Omega\colon \|\sigma\vec{Z}_m(\omega)\|_p 
              \leq \tfrac{1}{2}
       \right\},
\end{equation*}
defining the sub-probability measure
\begin{equation*}
  \mu(\,\cdot\,) := \P\left(\Omega' \cap \{\vec{X} \in \,\cdot\,\}\right).
\end{equation*}
Indeed, on $\Omega'$ we have
\begin{equation*}
  \|\vec{X}\|_p
    = \|\vec{u}_I + \sigma\vec{Z}_m\|_p
    \leq \|\vec{u}_I\|_p + \|\sigma\vec{Z}_m\|_p
    \stackrel{\omega \in \Omega'}{\leq}
      \tfrac{1}{2} + \tfrac{1}{2} = 1.
\end{equation*}
By Lemma~\ref{lem:truncationprob},
our choice of the scaling factor, $\sigma = \frac{1}{4\sqrt{p}}\, m^{-1/p}$,
implies that the truncation probability vanishes in the dimension limit:
\begin{equation} \label{eq:deltalim}
  \delta = \P\left(\|\sigma\vec{Z}_m\|_p > \frac{1}{2}\right)
    < \exp\left(-\frac{m^{2/p}}{\pi}\right)
    \xrightarrow[m\to\infty]{} 0.
\end{equation}
For this truncated measure, 
finding a lower bound for the $\mu$-average error of an algorithm $A \colon \R^m \to \R^m$ 
as defined in~\eqref{eq:emu} 
reduces to the problem of an approximate reconstruction of the random centres $\vec{u}_I$ under noise,
for $\eps > 0$ we have:
\begin{align}
  e^{\mu}(A,\ell_p^m \embed \ell_q^m)
    &= \expect\left[\|A(\vec{X}) - \vec{X}\|_q \cdot \ind_{\Omega'}\right] \nonumber\\
    &\geq \expect\left[\left(\|A(\vec{X}) - \vec{u}_I\|_q - \|\sigma \vec{Z}_m\|_q\right) 
                        \cdot \ind_{\Omega'}
                \right] 
        \nonumber \\
    &\geq \expect\left[\|A(\vec{X}) - \vec{u}_I\|_q \cdot \ind_{\Omega'}\right] 
      - \expect \|\sigma \vec{Z}_m\|_q \nonumber \\
    &\geq \bigl[\P(\|A(\vec{X}) - \vec{u}_I\|_q \geq \eps) - \delta\bigr] \cdot \eps
      - \expect \|\sigma \vec{Z}_m\|_q.
  \label{eq:emuvsmutilde}
\end{align}
We will choose $\eps = \frac{1}{4} k^{-\left(\frac{1}{p} - \frac{1}{q}\right)}$
and show that $\P(\|A(\vec{X}) - \vec{u}_I\|_q \geq \eps) > \frac{1}{16}$.
The term $\expect\|\sigma \vec{Z}_m\|_q$ turns out to be considerably smaller than $\eps$ for $k \ll m$,
eventually proving a $\mu$-average error of order $k^{-\left(\frac{1}{p} - \frac{1}{q}\right)}$,
see Lemma~\ref{lem:avgLB}.

\subsection{Reconstructing the centers from linear information}
\label{sec2.2}

In view of~\eqref{eq:emuvsmutilde}, 
we now focus on the probability of a large reconstruction error $\|A(\vec{X}) - \vec{u}_I\|_q$
for a deterministic algorithm~$A = \varphi \circ N\colon \R^m \to \R^m$ 
with non-adaptive linear information mapping $N\colon \R^m \to \R^n$.
We identify the information mapping with a matrix $N \in \R^{n \times m}$.
Without loss of generality,
the rows of $N$ are orthogonal with $\ell_2$-norm~$\sigma^{-1}$
such that $\vec{Z} := N(\sigma \vec{Z}_m) \sim \Normal(\vec{0},\id_n)$.
Hence,
the information $N(\vec{X})$ in $\R^n$ can be represented as a random vector
\begin{equation} \label{eq:Y}
    \vec{Y} := N(\vec{X}) = \vec{v}_I + \vec{Z}
\end{equation}
with $\vec{v}_i := N\vec{u}_i$ and independent random variables 
$I \sim \Uniform[M]$ and $\vec{Z} \sim \Normal(\vec{0},\id_n)$.

\begin{lemma} \label{lem:identify_I}
    Consider~$\vec{Y}$ as in~\eqref{eq:Y} with $n \geq 16$, and for $\eps > 0$ define
    \begin{equation} \label{eq:K}
      K := \sup_{\vec{x} \in \R^m} \#\{i \in [M] \colon \|\vec{x} - \vec{u}_i\|_q < \eps\}.
    \end{equation}
    Let $R > 0$ and define $\mathcal{I} := \{i \in [M] \colon \|\vec{v}_i\|_2 \leq R \}$.
    Assume that~$R$ satisfies $\# \mathcal{I} \geq \frac{M}{2}$ and $2R^2 \leq \log \frac{M}{8K}$.
    
    Then, for any reconstruction mapping $\varphi\colon \R^n \to \R^m$ one has
    \begin{equation*}
            \P\bigl(\|\varphi(\vec{Y)} - \vec{u}_I\|_q \geq \eps\bigr)
                \geq \frac{1}{16} \,. 
    \end{equation*}
\end{lemma}
\begin{proof}
    The random variable $\vec{Y}$ is distributed as a Gaussian mixture. Its density function is given by
    \begin{equation*}
        f_{\vec{Y}} = \frac{1}{M} \sum_{j = 1}^M f_j,
    \end{equation*}
    where $f_j$ is the density of a standard Gaussian random variable 
    centred at $\vec{v}_j$. For $i\in [M]$ we define the super-level sets
    \begin{equation*}
        L_{i,t} := \{\vec{y} \in \R^n \colon f_i(\vec{y}) \geq t\}
            = \left\{\vec{y} \in \R^n \colon 
                    \frac{1}{(2\pi)^{n/2}} \cdot \exp\left( - \frac{\|\vec{y} - \vec{v}_i\|^2}{2}\right) 
                        \geq t
                \right\} .
    \end{equation*}
    We extend this definition to $i=0$ with $\vec{v}_i := \vec{0}$ 
    and write $L_t := L_{0,t}$ for the super-level sets of a standard Gaussian random variable
    centred around the origin.
    For $t \in \left(0,(2\pi)^{-n/2}\right)$, 
    these super-level sets are Euclidean balls $B_{r(t)}(\vec{v}_i)$ 
    of radius $r(t) := \sqrt{-2\log\left((2\pi)^{n/2} t\right) }$ 
    centred around $\vec{v}_i$.
    For any measurable $D \subset \R^n$ 
    we can represent the following probability by an integral 
    over level sets of the probability density~$f_i$:
    \begin{align}
         \P(\vec{Y} \in D \,\wedge \, I = i)
             &= \frac{1}{M} \int_D f_{i}(\vec{y}) \, d\vec{y} 
                \nonumber\\
            &= \frac{1}{M} \int_{\R^n} \ind[\vec{y} \in D] 
                    \int_0^\infty \ind[f_i(\vec{y}) \geq t] \,dt \, d\vec{y} 
                \nonumber\\
            &= \frac{1}{M} \int_0^\infty \int_{\R^n} \ind_{D \cap L_{i,t}}(\vec{y}) \,d\vec{y} \,dt 
                \nonumber\\
            & = \frac{1}{M}\int_0^{\infty} \lambda^n( D  \cap L_{i,t} ) \, dt.
            \label{eq:YinD|i}
    \end{align}
    We define the regions of approximate detection of $\vec{u}_i$ by the reconstruction~$\varphi$,
    \begin{equation*}
        D_i := \{\vec{y} \in \R^n\colon \|\varphi(\vec{y}) - \vec{u}_i\|_q < \eps\},
    \end{equation*}
    The definition of $D_i$ and $K$ implies that, for all~$\vec{y} \in \R^n$,
    \begin{equation} \label{eq:yinDi}
        \#\{i \in [M]\colon \vec{y} \in D_i\} \leq K.
    \end{equation}
    We use these sets~$D_i$ to represent the probability of a bad approximation
    in the style of~\eqref{eq:YinD|i}:
    \begin{align}
        \P\bigl( \|\varphi(\vec{Y}) - \vec{u}_I\|_q \geq \eps \bigr)
            &= \frac{1}{M}\sum_{i = 1}^M \int_{0}^{\infty} \lambda^n(D_i^c \cap L_{i,t}) \, dt \nonumber\\
            &\geq \frac{1}{M} \sum_{i \in \mathcal{I}} \int_{0}^{\infty} 
                    \lambda^n(D_i^c \cap L_{i,t}) \, dt .
        \label{eq:P(badapprox)assum}
    \end{align}
    As a short-hand notation, we write $B_{\rho} := B_{\rho}(\vec{0}) \subset \R^n$.
    Let $V_n := \lambda^n(B_1)$ denote the volume of the $n$-dimensional Euclidean unit ball.
    If we set the radius to \mbox{$\rho(t) := \sqrt{r(t)^2 + R^2}$},
    then for $i \in \mathcal{I}$ the ball $B_{\rho(t)}$ 
    covers at least half of the volume of the super-level set $L_{i,t}$,
    see Figure~\ref{fig:ball-intersection}. 
    In detail,
    considering the half space
    \begin{equation*}
        H_i := \{\vec{y} \in \R^n \colon 
                    \langle \vec{y},\vec{v}_i \rangle \leq \langle\vec{v}_i,\vec{v}_i\rangle
                \},
    \end{equation*}
    we find
    \begin{equation} \label{eq:vol(LcapBrho)}
        \lambda^n(L_{i,t} \cap B_{\rho(t)})
            \geq \lambda^n\left(L_{i,t} \cap H_i\right)
            = \frac{1}{2} \, \lambda^n(L_{i,t}) = \frac{1}{2} \, r(t)^n \, V_n .
    \end{equation}
    Indeed: Elements $\vec{y} \in L_{i,t} \cap H_i$ can be written as orthogonal decomposition
    \mbox{$\vec{y} = \vec{u} + \vec{w}$}
    with 
    \mbox{$\vec{u} = \frac{\langle\vec{y},\vec{v}_i\rangle}{\langle\vec{v}_i,\vec{v}_i\rangle} \vec{v}_i$}.
    In the case of 
    \mbox{$\frac{\langle\vec{y},\vec{v}_i\rangle}{\langle\vec{v}_i,\vec{v}_i\rangle} \leq 0$}
    we have the estimate 
    \mbox{$\|\vec{u}\|_2 \leq \|\vec{u}\|_2 + \|\vec{v}_i\|_2 = \|\vec{u} - \vec{v}_i\|_2$}.
    Otherwise, since \mbox{$\vec{y} \in H_i$}, we have 
    \mbox{$\frac{\langle\vec{y},\vec{v}_i\rangle}{\langle\vec{v}_i,\vec{v}_i\rangle} \in [0,1]$}
    which implies \mbox{$\|\vec{u}\|_2 \leq \|\vec{v}_i\|_2$}.
    Therefore,
    \begin{align*}
        \|\vec{y}\|_2^2 
            &= \|\vec{u}\|_2^2 + \|\vec{w}\|_2^2 \\
            &\leq \max\{ \|\vec{v}_i\|_2^2, \|\vec{u} - \vec{v}_i\|_2^2 \} + \|\vec{w}\|_2^2 \\
            &\leq \|\vec{v}_i\|_2^2 + \|\vec{u} - \vec{v}_i\|_2^2 + \|\vec{w}\|_2^2 \\
            &= \|\vec{v}_i\|_2^2 + \|\vec{y} - \vec{v}_i\|_2^2 \\
            &\leq R^2 + r(t)^2 = \rho(t)^2,
    \end{align*}
    which shows $L_{i,t} \cap H_i \subset L_{i,t} \cap B_{\rho(t)}$.
    \begin{figure}
        \centering
        \begin{tikzpicture}[scale=0.6]
            
          \coordinate (O) at (0,0);
          \coordinate (V) at (3,0);
          \coordinate (Vtop) at ($(V)+(0,4)$);
          \coordinate (Vbot) at ($(V)+(0,-4)$);
          \coordinate (diag) at ($(O)!1!(Vtop)$); 
        
          \begin{scope}
            \clip (O) circle (5);
            \fill[black,opacity=0.1] (V) circle (4);
          \end{scope}
          \draw[thick, red] (O) circle (5);
          \draw[thick, blue] (V) circle (4);
        
          \node at (O) {$\bullet$};
          \node at (V) {$\bullet$};
          \node[below left] at (O) {$\vec{0}$};
          \node[below right] at (V) {$\vec{v}$};
        
          \draw[thick] (O) -- (V) node[midway,below] {$R$};
        
          \draw[thick, blue] (V) -- (Vtop) node[midway,right] {$r$};
        
          \draw[dashed] (V) -- (Vbot);

          \draw[thick] ($(V)+(-0.35,0)$) -- ($(V)+(-0.35,0.35)$) -- ($(V)+(0,0.35)$);
        
          \draw[thick, red] (O) -- (diag) node[midway,above left] {$\rho = \sqrt{r^2 + R^2}$};
        
        \end{tikzpicture}
        \caption{More than half the volume of $B_r(\vec{v})$ 
        lies inside of $B_{\rho}(\vec{0})$ if $\|\vec{v}\|_2 \leq R$ and $\rho = \sqrt{r^2 + R^2}$. 
        The figure depicts the extreme case of $\|\vec{v}\|_2 = R$.}
        \label{fig:ball-intersection}
    \end{figure}
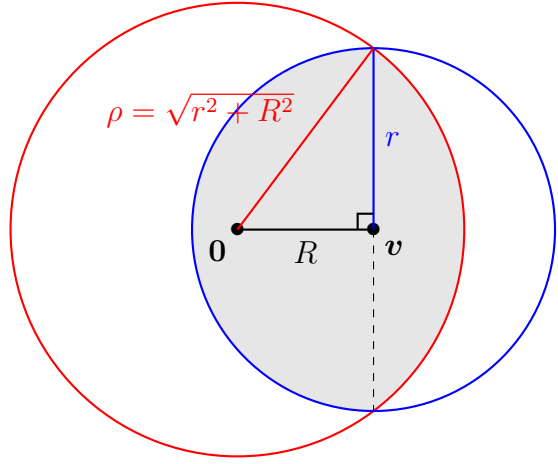
    Hence, \eqref{eq:vol(LcapBrho)} follows.
    
    Now we can estimate the sum that occurs in~\eqref{eq:P(badapprox)assum}:
    \begin{align}
        \sum_{i \in \mathcal{I}} \lambda^n(D_i^c \cap  L_{i,t})
            &\geq \sum_{i \in \mathcal{I}} \lambda^n(D_i^c \cap  L_{i,t} \cap B_{\rho(t)}) \nonumber\\
            &= \sum_{i \in \mathcal{I}} \lambda^n(L_{i,t} \cap B_{\rho(t)})
                - \sum_{i \in \mathcal{I}} \lambda^n(D_i \cap L_{i,t} \cap B_{\rho(t)}) \nonumber\\
        [\eqref{eq:yinDi}\text{ and } \eqref{eq:vol(LcapBrho)}]\qquad
            &\geq 
                \sum_{i \in \mathcal{I}} \frac{1}{2} \lambda^n(L_{i,t})
                    - K \cdot \lambda^n(B_{\rho(t)}) \nonumber\\
            &= \left(\frac{\# \mathcal{I}}{2} \cdot r(t)^n - K \cdot \rho(t)^n\right) V_n \nonumber\\
            &\geq \left(\frac{M}{4} \cdot r(t)^n - K \cdot \rho(t)^n\right) V_n\,. 
            \label{eq:Mrpow-Krhopow}
    \end{align}
    For $r(t) \geq \frac{\sqrt{n}}{2}$, 
    which is equivalent to $0 < t \leq T_n := (2\pi)^{-n/2} \cdot \exp\left(-\frac{n}{8}\right)$, 
    and using the assumption $2R^2 \leq \log\frac{M}{8K}$, we get
    \begin{align*}
        \frac{\rho(t)}{r(t)}
            &= \sqrt{1 + \frac{R^2}{r(t)^2}}
            \leq 1 + \frac{R^2}{2r(t)^2}
            \leq 1 + \frac{2R^2}{n}
            \leq 1 + \frac{\log \frac{M}{8K}}{n} \\
            &\leq \exp\left(\frac{\log \frac{M}{8K}}{n}\right)
            = \left(\frac{M}{8K}\right)^{1/n}.
    \end{align*}
    In that case, \eqref{eq:Mrpow-Krhopow} simplifies to
    \begin{equation*}
        \sum_{i \in \mathcal{I}} \lambda^n(D_i^c \cap  L_{i,t})
            \geq \frac{M}{8} \cdot r(t)^n \cdot V_n \,.
    \end{equation*}
    We finally find the desired lower estimate for the probability of a bad approximation
    as represented in~\eqref{eq:P(badapprox)assum}:
    \begin{align*}
        \P\bigl( \|\varphi(\vec{Y}) - \vec{u}_I\|_q \geq \eps \bigr)
            &\geq \frac{1}{M}\sum_{i \in \mathcal{I}} \int_{0}^{\infty} 
                    \lambda^n(D_i^c \cap L_{i,t}) \,dt \\
            &\geq \frac{1}{8} \int_0^{\infty} r(t)^n \, V_n \,dt \\
            &\geq \frac{1}{8} \int_0^{T_n} \left(r(t)^n - r(T_n)^n\right) V_n \, dt \\
        \left[r(T_n) = \tfrac{\sqrt{n}}{2}\right]\qquad
            &= \frac{1}{8} \int_0^{T_n} 
                    \lambda^n\left(
                        \left\{\vec{z}\in\R^n \colon \|\vec{z}\|_2 \geq \frac{\sqrt{n}}{2}
                        \right\} \cap L_{t}
                        \right) \, dt\\
        \left[\vec{Z} \sim \Normal(\vec{0},\id_n)\right]\qquad
            &= \frac{1}{8} \cdot \P\left(\|\vec{Z}\|_2 \geq \frac{\sqrt{n}}{2}\right) \\
            &\stackrel{\eqref{eq:P(|Z|>sqrt(n)/2}}{\geq}
                \frac{1}{16}\,.
    \end{align*}
    This finishes the proof.
\end{proof}

\subsection{Lower bound for the average case error}

\begin{lemma} \label{lem:avgLB}
    Let $2 \leq p < q \leq \infty$. Let $m,k \in \N$ with $m \geq 4^p$ and
    \begin{equation*}
        10 \leq k \leq \begin{cases}
                    c_{p,q} \cdot m
                        &\text{for } q < \infty, \\
                    c_{p,\infty} \cdot \frac{m}{(\log m)^{p/2}}
                        &\text{for } q = \infty,
                \end{cases}
    \end{equation*}
    with positive constants
    \begin{equation*}
        c_{p,q} = \begin{cases}
                    \left( \frac{1}{64} \sqrt{\frac{p}{q}}
                    \right)^{\left.1\middle/\left(\frac{1}{p} - \frac{1}{q}\right)\right.}
                        &\text{for } q < \infty, \\
                    \left(\frac{\sqrt{p}}{64 e}\right)^p
                        &\text{for } q = \infty.
                \end{cases}
    \end{equation*}
    Then for any non-adaptive deterministic algorithm $A$ that uses $n$ pieces of information
    where
    \begin{equation*}
        16 \leq n \leq \frac{1}{54p} \cdot k^{2/p} \cdot m^{1-\frac{2}{p}} \log \frac{m}{k},
    \end{equation*}
    we have the following lower bound for the $\mu$-average error
    for the sub-probability measure as defined in Section~\ref{sec:mu}:
    \begin{equation*}
        e^{\mu}(A,\ell_p^m \embed \ell_q^m)
            > \tfrac{1}{256} \,  k^{-\left(\frac{1}{p} - \frac{1}{q}\right)}.
    \end{equation*}
\end{lemma}
\begin{proof}
    We start by showing that our assumptions on $m$ and $k$ imply
    \begin{equation}\label{eq:m>>k}
        m > 64^2 \, k = 4096 \, k.
    \end{equation}
    For $q < \infty$, this follows directly if we bound $c_{p,q}$ from above
    using $\frac{p}{q} < 1$ and $\left(\frac{1}{p} - \frac{1}{q}\right) < \frac{1}{2}$.
    For $q = \infty$ we use $m \geq 4^p$ and obtain
    \begin{equation*}
        k \leq c_{p,\infty} \cdot \frac{m}{(\log m)^{p/2}}
          \leq \left(\frac{\sqrt{p}}{64e}\right)^p \cdot \frac{m}{(p \log 4)^{p/2}}
          = \left(\frac{1}{64e \sqrt{\log 4}}\right)^p \cdot m
          < \frac{1}{64^2} \cdot m.
    \end{equation*}
    
    We put $\eps = \frac{1}{4} \, k^{-\left(\frac{1}{p} - \frac{1}{q}\right)}$.
    For $\vec{x} \in \R^m$ consider the set
    \begin{equation*}
        J(\vec{x}) := \left\{i \in [M] \colon \|\vec{x} - \vec{u}_i\|_q < \eps\right\}.
    \end{equation*}
    If $J(\vec{x}) \neq \emptyset$, 
    for $j,i \in J(\vec{x})$ we have $\|\vec{u}_i - \vec{u}_j\|_q < 2\eps$,
    hence, for any fixed $j \in J(\vec{x})$ we obtain the inclusion
    \begin{equation*}
        J(\vec{x}) \subseteq \left\{i \in [M] \colon \|\vec{u}_j - \vec{u}_i\|_q < 2\eps\right\}.
    \end{equation*}
    Recalling $\vec{u}_i = \frac{1}{2} k^{-1/p} \vec{w}_i$, we can apply Corollary~\ref{cor:U(v)} 
    (the prerequisite $m \geq 3k$ is covered by~\eqref{eq:m>>k}) and obtain
    \begin{align*}
        \#J(\vec{x}) 
        &\leq \# \left\{i \in [M] \colon \|\vec{u}_j - \vec{u}_i\|_q < 2\eps\right\} \\
        &= \#\left\{i \in [M] \colon \|\vec{w}_j - \vec{w}_i\|_q < \tfrac{1}{2} \, k^{1/q}\right\} \\
        &< \left(\frac{k}{m}\right)^{k/10} \cdot \# W\,.
    \end{align*}
    Hence, with $K = \sup\limits_{\vec{x} \in \R^m} \#J(\vec{x})$ as defined in~\eqref{eq:K} 
    and $M = \# W$, we find
    \begin{equation} \label{eq:M/K}
        \frac{M}{K} > \left(\frac{m}{k}\right)^{k/10}.
    \end{equation}
    The prerequisite $2R^2 \leq \log\frac{M}{8K}$ from Lemma~\ref{lem:identify_I} 
    will be met if
    \begin{equation*}
        R^2 \leq \frac{1}{2}\left(\frac{k}{10} \log \frac{m}{k} - \log 8\right).
    \end{equation*}
    Since $k \geq 10$ and $\frac{m}{k} > 4096$, see~\eqref{eq:m>>k}, 
    the following is an even stricter requirement:
    \begin{equation} \label{eq:Rsecondrequirement}
        R^2 \leq \frac{k}{27} \log \frac{m}{k}\,.
    \end{equation}

    Recall $\vec{v}_i = N\vec{u}_i = \frac{1}{2} \, k^{-1/p} \cdot N\vec{w}_i$
    from~\eqref{eq:Xdef} and~\eqref{eq:Y},
    where in the context of Lemma~\ref{lem:identify_I} we assumed that
    the information matrix $N = (N_{ij}) \in \R^{n \times m}$ has orthogonal rows
    with $\ell_2$-norm $\sigma^{-1} = 2\sqrt{p}\,m^{1/p}$ each.
    The coordinates of \mbox{$\vec{w}_I = (X_1,\ldots,X_m)$} 
    take values in~$\{-1,0,1\}$ and are pairwise uncorrelated,
    hence,
    \begin{align*}
        \expect \|\vec{v}_I\|_2^2
            &= \tfrac{1}{4} \, k^{-2/p} \sum_{i=1}^n \expect\left(\sum_{j=1}^m N_{ij} X_j\right)^2 \\
            &= \tfrac{1}{4} \, k^{-2/p} \sum_{i=1}^n \sum_{j=1}^m 
                    N_{ij}^2 \underbrace{\expect X_j^2}_{=\frac{k}{m}} \\
            &= \tfrac{1}{4} \, k^{-2/p} \cdot \frac{k}{m} \cdot n \cdot 4p\,m^{2/p} \\
            &= p n \cdot \left(\frac{k}{m}\right)^{1-\frac{2}{p}}.
    \end{align*}
    Putting $R:= \sqrt{2pn} \cdot \left(\frac{k}{m}\right)^{\frac{1}{2} - \frac{1}{p}}$,
    Markov's inequality gives $\P\left(\|\vec{v}_I\|_2^2 \leq R^2\right) \geq \frac{1}{2}$, 
    in other words,
    the set $\mathcal{I}$ as defined in Lemma~\ref{lem:identify_I} 
    satisfies $\#\mathcal{I} \geq \frac{M}{2}$.
    For this value of $R$, the requirement~\eqref{eq:Rsecondrequirement} reads
    \begin{equation*}
        2pn \cdot \left(\frac{k}{m}\right)^{1-\frac{2}{p}}
            \leq \frac{k}{27} \log \frac{m}{k}
        \qquad\Longleftrightarrow\qquad
        n \leq \frac{1}{54p} \cdot k^{2/p} \cdot m^{1-\frac{2}{p}} \log \frac{m}{k} \,.
    \end{equation*}
    With this, the assumptions of Lemma~\ref{lem:identify_I} are thus satisfied,
    which implies the probability bound $\P(\|A(\vec{X}) - \vec{u}_I\|_q \geq \eps) \geq \frac{1}{16}$.
    Now, we can finally evaluate the lower bound~\eqref{eq:emuvsmutilde} for the $\mu$-average error:
    \begin{equation*}
        e^{\mu}(A,\ell_p^m \embed \ell_q^m)
            \geq \bigl[\P(\|A(\vec{X}) - \vec{u}_I\|_q \geq \eps) - \delta\bigr] \cdot \eps
              - \expect \|\sigma \vec{Z}_m\|_q.
    \end{equation*}
    For $m \geq 4^p > \left(\pi \log(32)\right)^{p/2}$
    we have 
    $\delta \stackrel{\eqref{eq:deltalim}}{<} \exp\left(-\frac{m^{2/p}}{\pi}\right) < \frac{1}{32}$, 
    hence,
    \begin{equation*}
        \bigl[\P(\|A(\vec{X}) - \vec{u}_I\|_q \geq \eps) - \delta\bigr] \cdot \eps
            > \left[\tfrac{1}{16} - \tfrac{1}{32}\right] 
                \cdot \tfrac{1}{4} \, k^{-\left(\frac{1}{p} - \frac{1}{q}\right)}
            = \tfrac{1}{128} \, k^{-\left(\frac{1}{p} - \frac{1}{q}\right)}.
    \end{equation*}
    From~\eqref{eq:GaussianExBound} and~\eqref{eq:GaussianExBoundoo}, 
    which can be used since $m \geq 3$,
    we have
    \begin{equation*}
      \expect \|\sigma \vec{Z}_m\|_q
        \leq \begin{cases}
                \frac{1}{4} \sqrt{\frac{q}{p}} \cdot  m^{-\left(\frac{1}{p} - \frac{1}{q}\right)}
                  & \text{for } q < \infty, \\
                \frac{e}{4} \sqrt{\frac{\log m}{p}} \cdot m^{-1/p}
                  & \text{for } q = \infty.
              \end{cases}
    \end{equation*}
    Restricting the range of $k$ relative to $m$ as stated in the assumptions of the lemma
    ensures 
    $\expect \|\sigma \vec{Z}_m\|_q \leq \frac{1}{256} \,  k^{-\left(\frac{1}{p} - \frac{1}{q}\right)}$,
    hence,
    \begin{equation*}
        e^{\mu}(A,\ell_p^m \embed \ell_q^m)
            > \tfrac{1}{256} \,  k^{-\left(\frac{1}{p} - \frac{1}{q}\right)}.
    \end{equation*}
    This is the statement of the lemma.
\end{proof}

\section{Main result}
\label{sec:rannonadaLB}

We employ the techniques from the previous section to obtain asymptotic results on the vector approximation problem $\ell_p^m \embed \ell_q^m$.
Error bounds are given in terms of a joint $(m,n)$-asymptotic,
complexity bounds in terms of a joint $(m,\eps)$-asymptotic, compare the last paragraph of Section~\ref{sec:Intro}.

\begin{theorem}\label{thm:Main}
    Let $2 \leq p < q \leq \infty$ and $m,n \in \N$ where $m > m_0(p,q)$ is sufficiently large.
    Restricting to $n \leq m^\alpha$ for some $0 < \alpha < 1$,
    we obtain non-adaptive randomized error rates
    \begin{equation*}
        e^{\ran,\nonada}(n,\ell_p^m \embed \ell_q^m)
            \asymp \min \left\{  1,\; 
                            \left( \frac{m^{1-\frac{2}{p}} \cdot \log m}{n} 
                            \right)^{\frac{1}{2}\left(1-\frac{p}{q}\right)}  
                        \right\}. 
    \end{equation*}
    Conversely, let $0 < \beta < \frac{1}{p} - \frac{1}{q}$.
    Then there exists a constant $\eps_0 \in (0,1)$
    such that for $\eps \in [m^{-\beta},\eps_0]$,
    we have the non-adaptive randomized asymptotic complexity
    \begin{equation*}
        n^{\ran,\nonada}(\eps, \ell_p^m \embed \ell_q^m)
            \asymp \eps^{\left. - \frac{2}{p} \middle/ \left(\frac{1}{p} - \frac{1}{q}\right)\right.} 
                \cdot m^{1-\frac{2}{p}} \cdot \log m.
    \end{equation*}
    The implicit constants may depend on $p$, $q$, $\alpha$, and $\beta$,
    but not on $m$ and $n$ or $\eps$, respectively.
\end{theorem}
\begin{proof}
    The upper error and complexity bounds can be found in \cite[Thm~4.4]{KW24c}.

    We start with the lower bound for the complexity.
    We switch to the average case setting with
    \begin{equation*}
        k := \left\lfloor (256\, \eps)^{-\left.1\middle/\left(\frac{1}{p} - \frac{1}{q}\right)\right.} 
                \right\rfloor
            \asymp \eps^{-\left.1\middle/\left(\frac{1}{p} - \frac{1}{q}\right)\right.}.
    \end{equation*}
    We have $k \geq 10$ 
    if we assume $\eps \leq \frac{1}{256} \cdot 10^{-\left(\frac{1}{p} - \frac{1}{q}\right)} =: \eps_1$.
    Restricting to
    \begin{equation} \label{eq:eps>precise}
        \eps 
            \geq \begin{cases}
                    \frac{1}{4}\sqrt{\frac{q}{p}} \cdot m^{-\left(\frac{1}{p} - \frac{1}{q}\right)}
                        &\text{for } q < \infty,\\
                    \frac{e}{4\sqrt{p}} \cdot m^{-1/p} \sqrt{\log m}
                        &\text{for } q = \infty,
                \end{cases}
    \end{equation}
    the conditions for $k \ll m$ from Lemma~\ref{lem:avgLB} are met.
    For $\eps \geq m^{-\beta}$ and sufficiently large $m$,
    the prerequisite~\eqref{eq:eps>precise} is fulfilled.
    If we take
    \begin{align} \label{eq:ntoosmall}
        16 \leq 
        n &\leq \frac{1}{54p} \cdot k^{2/p} \cdot m^{1-\frac{2}{p}} \cdot \log \frac{m}{k} \\
            &\asymp \eps^{-\left.\frac{2}{p}\middle/\left(\frac{1}{p} - \frac{1}{q}\right)\right.}
                \cdot m^{1-\frac{2}{p}} 
                \cdot \log \left(m \cdot 
                            \eps^{\left.1\middle/\left(\frac{1}{p} - \frac{1}{q}\right)\right.}
                        \right),
            \nonumber
    \end{align}
    then Bakhvalov's technique~\eqref{eq:Bakhvalov} and Lemma~\ref{lem:avgLB} imply
    \begin{equation*}
        e^{\ran,\nonada}(n,\ell_p^m \embed \ell_q^m)
            \geq \inf_{A_n} e^{\mu}(A_n, \ell_p^m \embed \ell_q^m)
            > \eps,
    \end{equation*}
    where the infimum is taken over non-adaptive deterministic algorithms~$A_n$
    that use at most $n$~pieces of linear information.
    This argument only works if there exists $n \in \N$ that fulfils~\eqref{eq:ntoosmall}.
    Since $\frac{m}{k} \geq 4096$ for $k$ satisfying the conditions of Lemma~\ref{lem:avgLB}, 
    see~\eqref{eq:m>>k}, we have
    \begin{equation*}
        \frac{1}{54p} \cdot k^{2/p} \cdot m^{1-\frac{2}{p}} \cdot \log \frac{m}{k} 
        \geq \frac{2 \log 2}{9 p} \cdot 4096^{1-\frac{2}{p}} \cdot k.
    \end{equation*}
    The right-hand side is larger than $16$ for
    $k \geq k_0 := \left\lceil \frac{72p}{\log 2} \cdot 4096^{-\left(1-\frac{2}{p}\right)}\right\rceil$.
    Assuming $\eps \leq \frac{1}{256} \, k_0^{-\left(\frac{1}{p} - \frac{1}{q}\right)} =: \eps_2$,
    we have $k \geq k_0$.
    Altogether, $\eps_0 := \min\{\eps_1,\eps_2\}$ is a suitable upper limit for~$\eps$.
    
    The conclusion of Lemma~\ref{lem:avgLB} then is 
    that if we want to achieve an error at most $\eps$,
    we need to take more measurements than in~\eqref{eq:ntoosmall},
    that means
    \begin{equation*}
        n^{\ran,\nonada}(\eps,\ell_p^m \embed \ell_q^m)
            \succeq \eps^{-\left.\frac{2}{p}\middle/\left(\frac{1}{p} - \frac{1}{q}\right)\right.}
                \cdot m^{1-\frac{2}{p}} 
                \cdot \log \left(m \cdot 
                        \eps^{\left.1\middle/\left(\frac{1}{p} - \frac{1}{q}\right)\right.}
                    \right).
    \end{equation*}
    In summary, there exists a constant $c = c(p,q) > 0$
    such that for $\eps$ satisfying~\eqref{eq:eps>precise} and $\eps \leq \eps_0$,
    we have
    \begin{equation} \label{eq:n(eps,m)}
        n^{\ran,\nonada}(\eps,\ell_p^m \embed \ell_q^m)
            > c \cdot \eps^{-\left.\frac{2}{p}\middle/\left(\frac{1}{p} - \frac{1}{q}\right)\right.}
                \cdot m^{1-\frac{2}{p}} 
                \cdot \log \left(m \cdot 
                        \eps^{\left.1\middle/\left(\frac{1}{p} - \frac{1}{q}\right)\right.}
                    \right)
            =: \nu(\eps).
    \end{equation}
    Define
    \begin{equation*}
        n_0 := \lfloor \nu(\eps_0) \rfloor
            \succeq m^{1-\frac{2}{p}} \cdot \log m .
    \end{equation*}
    Due to monotonicity of the $n$-th minimal error, for $n \leq n_0$ we have
    \begin{equation*}
        e^{\ran,\nonada}(n,\ell_p^m \embed \ell_q^m)
            > e^{\ran,\nonada}(n_0,\ell_p^m \embed \ell_q^m) \geq \eps_0 \asymp 1.
    \end{equation*}
    For $n > n_0$, put
    \begin{equation} \label{eq:e(n)def}
        e(n) := \left(\frac{cp}{2} \cdot \frac{m^{1-\frac{2}{p}} \cdot \log \frac{m}{n}}{n}
                \right)^{\frac{p}{2}\left(\frac{1}{p} - \frac{1}{q}\right)},
    \end{equation}
    which is chosen such that, 
    for $m \geq e^{2/(cp)} \cdot n \;\Longleftrightarrow\; \frac{cp}{2} \cdot \log \frac{m}{n} \geq 1$, 
    we have
    \begin{equation} \label{eq:nu(e(n))}
        \nu(e(n)) 
            = \frac{n}{\log\frac{m}{n}} 
                \cdot \log\left(\frac{cp}{2} \cdot \frac{m \log \frac{m}{n}}{n}\right)
            \geq n.
    \end{equation}
    Assuming 
    $m \geq \max\{10,k_0\} \cdot 256^{\left.1\middle/\left(\frac{1}{p} - \frac{1}{q}\right)\right.} \cdot e^{1/c} \;\Longleftrightarrow\; c \cdot \log \left(m \cdot \eps_0^{\left.1\middle/\left(\frac{1}{p} - \frac{1}{q}\right)\right.}\right) \geq 1$,
    and using the fact that the expression $e(n)$ defined in~\eqref{eq:e(n)def} 
    is monotonically decaying in $n$ on its natural domain~$(0,m)$,
    for $n \geq \nu(\eps_0)$ we have
    \begin{align*}
        e(n) &\leq e(\nu(\eps_0)) \\
            &= \eps_0 \cdot \left(
                    \log \left(
                        \frac{m \cdot 
                            \eps_0^{\left.1\middle/\left(\frac{1}{p} - \frac{1}{q}\right)\right.}
                        }{\left(c \cdot \log \left(
                            m \cdot \eps_0^{\left.1\middle/\left(\frac{1}{p} - \frac{1}{q}\right)\right.}
                                \right)
                        \right)^{\frac{p}{2}}}
                    \right)
                \cdot \frac{1}{
                    \log \left(m 
                            \cdot \eps_0^{\left.1\middle/\left(\frac{1}{p} - \frac{1}{q}\right)\right.}
                        \right)}
                            \right)^{\frac{p}{2}\left(\frac{1}{p} - \frac{1}{q}\right)} \\
            &\leq \eps_0.
    \end{align*}
    Furthermore, assuming $m \geq en \;\Longleftrightarrow\; \log\frac{m}{n} \geq 1$,
    the expression~\eqref{eq:e(n)def} is lower bounded by
    \begin{equation*}
        e(n) \geq \left(\frac{cp}{2} \cdot \frac{m}{n}
                \right)^{\frac{p}{2}\left(\frac{1}{p} - \frac{1}{q}\right)}
                \cdot m^{-\left(\frac{1}{p} - \frac{1}{q}\right)}.
    \end{equation*}
    Hence, we can guarantee that $\eps = e(n)$ satisfies the lower constraint~\eqref{eq:eps>precise}
    if we further assume
    \begin{align}
        m &\geq \frac{2}{cp} \left(\frac{1}{4}\sqrt{\frac{q}{p}}\right)^{
                    \left.\frac{2}{p}\middle/\left(\frac{1}{p} - \frac{1}{q}\right)\right.
                }
                \cdot n
            &&\text{for } q < \infty, \\
        \frac{m}{\log m} &\geq \frac{e^2}{8cp^2} \cdot n
            &&\text{for } q = \infty.
    \end{align}
    Restricting to $n \leq m^\alpha$ with some $\alpha \in (0,1)$,
    these conditions are indeed satisfied for large~$m$.
    Then, by~\eqref{eq:n(eps,m)} and the definition of $e(n)$, for sufficiently large $m$, we thus have
    \begin{equation*}
        n^{\ran,\nonada}(e(n), \ell_p^m \embed \ell_q^m) > \nu(e(n)) 
            \stackrel{\eqref{eq:nu(e(n))}}{\geq} n.
    \end{equation*}
    Hence,
    \begin{equation} \label{eq:e(n)asymp}
        e^{\ran,\nonada}(n,\ell_p^m \embed \ell_q^m)
            > e(n) \asymp \left(\frac{m^{1-\frac{2}{p}} \cdot \log \frac{m}{n}}{n}
                \right)^{\frac{p}{2}\left(\frac{1}{p} - \frac{1}{q}\right)}.
    \end{equation}

    Finally, restricting ourselves to $\eps \geq m^{-\beta}$, we have
    \begin{equation*} \label{eq:asymplog(m)}
        \left(1 - \frac{\beta}{\frac{1}{p} - \frac{1}{q}}\right) \log m
            \leq \log \left(m \cdot  \eps^{\left.1\middle/\left(\frac{1}{p} - \frac{1}{q}\right)\right.}
                        \right)
            \leq \log m,
    \end{equation*}
    hence, the logarithmic term in~\eqref{eq:n(eps,m)} is asymptotically equivalent to $\log m$.
    Analogously, for $n \leq m^{\alpha}$ with some $\alpha \in (0,1)$,
    we have $\log\frac{m}{n} \asymp \log m$, thus simplifying~\eqref{eq:e(n)asymp}.
\end{proof}

\begin{remark}
    Compared to the upper bound results from~\cite[Thm~4.4]{KW24c},
    the range of parameters for which we prove the lower bounds
    is more constrained here.
    However, these restrictions are quite natural.
    
    Assume, for example, that we had 
    a very small value $\eps \preceq m^{-\left(\frac{1}{p} - \frac{1}{q}\right)}$.
    Then the expression of the complexity bound would evaluate to
    \begin{equation*}
        \eps^{\left. - \frac{2}{p} \middle/ \left(\frac{1}{p} - \frac{1}{q}\right)\right.} 
                \cdot m^{1-\frac{2}{p}} \cdot \log m
            \succeq m \cdot \log m ,
    \end{equation*}
    but $n=m$ already means that we have complete information about~$\vec{x} \in \R^m$, 
    allowing for exact recovery.

    For $\eps \geq \eps_0$ the expression for the complexity bound would give a value
    \begin{equation*}
         \eps^{\left. - \frac{2}{p} \middle/ \left(\frac{1}{p} - \frac{1}{q}\right)\right.} 
                \cdot m^{1-\frac{2}{p}} \cdot \log m
            \preceq m^{1 - \frac{2}{p}} \cdot \log m .
    \end{equation*}
    Restricting to such small values $n \preceq m^{1 - \frac{2}{p}} \cdot \log m$,
    the error bound gives
    \begin{equation*}
        e^{\ran,\nonada}(n,\ell_p^m \embed \ell_q^m) \asymp 1,
    \end{equation*}
    meaning that we are essentially not improving over the initial error 
    \begin{equation*}
        e^{\ran,\nonada}(0,\ell_p^m \embed \ell_q^m) = 1.
    \end{equation*}
\end{remark}

\section{Conclusions and outlook}

Previously known asymptotic lower bounds due to Heinrich~\cite{H92}
state that for \mbox{$m \geq 2n$} and $n \geq 2$ we have
\begin{equation} \label{eq:He92LB}
    e^{\ran}(n, \ell_p^m \embed \ell_q^m)
        \succeq \begin{cases}
                    n^{-\left(\frac{1}{p} - \frac{1}{q}\right)}
                        &\text{for } 1 \leq p < q < \infty,\\
                    n^{-1/p} \cdot \sqrt{\log n}
                        &\text{for } 1 \leq p < q = \infty.
                \end{cases}
\end{equation}
These bounds were obtained by studying the average case setting
with a truncated Gaussian measure centered around the origin.
By studying an average case setting 
for a truncated Gaussian mixture with centers around $k$-sparse vectors,
we now obtained lower bounds that not only exhibit a dependence on $m$,
but also match the improved upper bounds
from our recent paper~\cite{KW24b}.
The older $m$-independent bounds~\eqref{eq:He92LB} 
hold for the entire parameter range $1 \leq p < q \leq \infty$
and general randomized algorithms, including adaptive ones,
while our new $m$-dependent bounds
only cover the case of non-adaptive randomized methods for $p \geq 2$.

We believe that the same average case setting with Gaussian mixtures
can be used to prove lower bounds for adaptive randomized methods 
in the regime $2 \leq p < q \leq \infty$,
namely by employing techniques from information theory in the spirit of~\cite{PW12}
(a paper that studies lower bounds for adaptive randomized methods
in the related field of sparse recovery).
However, if we want to cover the case $1 \leq p < 2$,
we expect that a more complicated Gaussian measure is needed.
In detail, instead of a random vector $\vec{X}$ as defined in~\eqref{eq:Xdef},
we consider a random vector of the shape
\begin{equation*}
    \vec{X} = \vec{u}_I + \sigma P\vec{Z}_m,
\end{equation*}
where $P$ is a random projection onto $2n$~coordinates independent of $I$ and $\vec{Z}_m$.
(Without such a random projection, for $p < 2$ the scaling factor $\sigma$ would need to be too small
in order to control the truncation probability.)
A~measure like this has been studied in~\cite{KNW24}, yet only for $1$-sparse vectors $\vec{u}_i$.

\section*{Dedication}

We dedicate this paper to Henryk Wo\'{z}niakowski on the occasion of his 80th birthday.
He, as a leading example for the study 
of average-case and randomized settings in information-based complexity 
with deep attention to the impact of the problem size on error and cost bounds,
is a huge inspiration for our own work as in this article.
Dear Henryk, we wish you good health in the years to come
with good company and fruitful discussions with your dear colleagues!

\appendix

\section{Gaussian measures in high dimensions}

Throughout this section
let $\vec{Z}$ denote a standard Gaussian random vector with values in $\R^m$.
For $1 \leq p < \infty$ we have
\begin{equation}\label{eq:GaussianExBound}
    \sqrt{\frac{2}{\pi}} \cdot m^{1/p} 
        \leq \expect \|\vec{Z}_m\|_p 
        \leq \sqrt{p} \cdot m^{1/p},
\end{equation}
see Pisier~\cite[Lem~4.14]{Pis86} for the initial result and \cite[Lem~A.9]{Ku17} for explicit constants.
With $p=\log m$, for $m \geq e$, we obtain
\begin{equation}\label{eq:GaussianExBoundoo}
  \expect \|\vec{Z}_m\|_\infty
    \leq \expect \|\vec{Z}_m\|_p
    \leq e \cdot \sqrt{\log m}
\end{equation}

The following lemma is about concentration of Gaussian vectors around the expectation of their norm,
see~\cite[pp.~180/181]{Pis86}, or \cite[Lem~2.1.6]{adler2009random}
for a more general statement.
Here, we state it specifically applied to the $\ell_p$-norm with $p \geq 2$.

\begin{lemma}\label{lem:GaussianConcentration}
    For $2 \leq p < \infty$ we have
    \begin{align*}
        \P\bigl( \|\vec{Z}_m\|_p > (1+t) \expect \|\vec{Z}_m\|_p\bigr) 
            &\leq \exp\left(- \frac{t^2 m^{2/p}}{\pi}\right), \\
        \P\left( \|\vec{Z}_m\|_p < (1-t) \expect \|\vec{Z}_m\|_p\right) 
            &\leq  \exp\left(- \frac{t^2 m^{2/p}}{\pi}\right).
    \end{align*}
\end{lemma}
\begin{proof}
    The result from \cite[pp.~180/181]{Pis86} applied to the embedding $\ell_2^m \embed \ell_p^m$ states
    \begin{equation*}
        \P\bigl( \|\vec{Z}_m\|_p > (1+t) \expect \|\vec{Z}_m\|_p\bigr) 
            \leq \exp\left(-\frac{t^2 \varrho^2}{2}\right),
    \end{equation*}
    where
    $\varrho := \expect\|\vec{Z}_m\|_p / \|\id\|_{\ell_2^m\to \ell_p^m}$.
    The same bound holds for the second probability.
    For $p \geq 2$ we have the embedding constant $\|\id\|_{\ell_2^m\to \ell_p^m} = 1$,
    so by the first inequality of~\eqref{eq:GaussianExBound} 
    we have $\rho \geq \sqrt{\frac{2}{\pi}} \cdot m^{1/p}$,
    which gives
    \begin{equation*}
        \exp\left(-\frac{t^2 \varrho^2}{2}\right) \leq \exp\left(- \frac{t^2 m^{2/p}}{\pi}\right).
    \end{equation*}
    This bounds the right-hand side of both inequalities as stated.
\end{proof}

In particular, for $p = 2$ and $t = 1 - \sqrt{\frac{\pi}{8}}$ we have
\begin{align}
    \P\left(\|\vec{Z}_m\|_2 \geq \frac{\sqrt{m}}{2}\right)
        &= 1 - \P\left( \|\vec{Z}_m\|_2 < (1-t) \cdot \sqrt{\frac{2m}{\pi}}\right) \nonumber \\
        &\geq 1 - \P\bigl( \|\vec{Z}_m\|_2 < (1-t) \expect \|\vec{Z}_m\|_2\bigr) \nonumber\\
        &\geq 1 - \exp\left(-\frac{t^2 m}{\pi}\right)
        \xrightarrow[m \to \infty]{} 1.
        \label{eq:gauss-without-inner-ball}
\end{align}
Exploiting that the last expression in~\eqref{eq:gauss-without-inner-ball}
is monotone in~$m$, we find
\begin{equation} \label{eq:P(|Z|>sqrt(n)/2}
    \P\left(\|\vec{Z}_m\|_2 \geq \tfrac{\sqrt{m}}{2}\right) \geq \tfrac{1}{2}
    \qquad\text{for } m \geq 16.
\end{equation}
This estimate will be crucial in the proof of Lemma~\ref{lem:identify_I}.
The next case is important to bound the truncation probability of the measure 
we define in Section~\ref{sec:mu}.

\begin{lemma} \label{lem:truncationprob}
    Let $2 \leq p < \infty$.
    For $\epsilon > 0$ put $\sigma:= \sigma(\epsilon) = \frac{\epsilon}{2\sqrt{p}} \cdot m^{-1/p}$.
    Then
    \begin{equation*}
        \P\left( \| \sigma \vec{Z}_m \|_p > \epsilon \right) 
            < \exp\left( - \frac{ m^{2/p}}{\pi} \right).
    \end{equation*}
\end{lemma}
\begin{proof}
    We apply Lemma~\ref{lem:GaussianConcentration} with $t=1$:
    \begin{align*}
        \P\left( \| \sigma \vec{Z}_m \|_p > \epsilon \right)
            &= \P\left( \|\vec{Z}_m\|_p > 2\sqrt{p} \, m^{1/p} \right) \\
            &\stackrel{\eqref{eq:GaussianExBound}}{\leq}
                \P\left( \|\vec{Z}_m\|_p > 2 \expect \|\vec{Z}_m\|_p \right) \\
            &\leq \exp\left(-\frac{m^{2/p}}{\pi}\right). \qedhere
    \end{align*}
\end{proof}

\section{Some coding theory}

We apply a result from coding theory about the existence of codes 
known as Gilbert-Varshamov bounds, compare~\cite[Lem~3.1, Sec~A]{BIPW10}.
For a general introduction to coding theory,
see, for instance~\cite{vL98}.
For words $\boldsymbol{u},\boldsymbol{v} \in [q]^k$ of length $k$ written in an alphabet of length $q$,
their \emph{Hamming distance} is defined as
\begin{equation*}
    \Delta(\vec{u},\vec{v}) := \# \{i \in [k] \colon u_i \neq v_i\} .
\end{equation*}
For a subset $T \subseteq [q]^k$, called a \emph{codebook}, 
we define its \emph{minimal Hamming distance} by
\begin{equation*}
    \Delta(T) := \min_{\substack{\vec{u},\vec{v} \in T \\ \vec{u} \neq \vec{v}}} \Delta(\vec{u},\vec{v}) .
\end{equation*}
The following lemma is standard in coding theory.
It is based on a special case of~\cite[Claim~A.1]{BIPW10}.

\begin{lemma} \label{lem:Gilbert-Varshamov}
    Let $q \geq 6$ and $k$ be natural numbers.
    Then there exists a codebook $T \subseteq [q]^k$ 
    with minimal Hamming distance $\Delta(T) \geq \frac{k}{2}$ and size
    \begin{equation*}
        \# T > q^{k/10} .
    \end{equation*}
\end{lemma}
\begin{proof}
    We define $r := \left\lfloor \frac{k-1}{2} \right\rfloor$.
    For $\vec{u} \in [q]^k$ we define the Hamming ball of radius $r$ by
    \begin{equation*}
        B_r(\vec{u}) := \{\vec{v} \in [q]^k \colon \Delta(\vec{u},\vec{v}) \leq r\} \,.
    \end{equation*}
    By the definition of $r$, 
    for $\vec{v} \in [q]^k \setminus B_r(\vec{u})$ we have $\Delta(\vec{u},\vec{v}) \geq \frac{k}{2}$. 
    Assume that we have a codebook $T \subseteq [q]^k$ with $\Delta(T) \geq \frac{k}{2}$ that satisfies
    $\# T < q^k / \# B_r(\vec{u})$,
    then there exists $\vec{v} \in [q]^k \setminus \bigcup_{\vec{u} \in T} B_r(\vec{u})$ 
    and the larger codebook still satisfies $\Delta(T \cup \{\vec{v}\}) \geq \frac{k}{2}$.
    This packing argument implies
    that there exist $T \subseteq [q]^k$ with $\Delta(T) \geq \frac{k}{2}$ and
    \begin{equation} \label{eq:packingargument}
        \# T \geq \frac{q^k}{\# B_r(\vec{u})} \,.
    \end{equation}
    We estimate the cardinality of a single Hamming ball in $[q]^k$:
    \begin{align*}
        \# B_r(\vec{u})
            &= \sum_{j=0}^r \binom{k}{j} (q-1)^j
            \leq (q-1)^r \sum_{j=0}^r \binom{k}{j}
            \leq (q-1)^r \cdot \frac{1}{2} \underbrace{\sum_{j=0}^k \binom{k}{j}}_{=(1+1)^k} \\
            &\leq q^{k/2} \cdot 2^k.
    \end{align*}
    Together with~\eqref{eq:packingargument} we have
    \begin{equation*}
        \#T \geq q^{k/2} \cdot 2^{-k}
            = q^{\left(\frac{1}{2} - \log_q 2\right) k}
            \stackrel{[q \geq 6]}{>} q^{k/10} . \qedhere
    \end{equation*}
\end{proof}

We use this result to prove a lower for a packing number in $W = W_k^m$ as defined in~\eqref{eq:W_k}.
The following is adapted from~\cite[Lem~3.1]{BIPW10}.
\begin{lemma} \label{lem:SsubsetW}
    Let $k,m \in \N$ with $m \geq 3k$ and $p \in [1,\infty]$. 
    Then there exists a set $S \subseteq W_k^m$ with cardinality
    \begin{equation*}
        \# S \geq \left(2\left\lfloor \frac{m}{k} \right\rfloor\right)^{k/10}
            > \left(\frac{m}{k}\right)^{k/10} .
    \end{equation*}
    that has a minimal $\ell_p$-separation of
    \begin{equation*}
        \min_{\substack{\vec{v},\vec{v}'\in S \\ \vec{v} \neq \vec{v'}}} \|\vec{v} - \vec{v}'\|_p
            \geq k^{1/p}.
    \end{equation*}
\end{lemma}
\begin{proof}
    Consider the alphabet
    $A = \{\vec{e}_1,\ldots,\vec{e}_s,-\vec{e}_1,\ldots,-\vec{e}_s\} \subset \R^s$
    for~$s := \left\lfloor \frac{m}{k} \right\rfloor$
    with $q := \#A = 2s \geq 6$ letters.
    An element $\vec{u} = (a_1,\ldots,a_k) \in A^k$ corresponds to a concatenated vector
    \begin{equation*}
        \vec{v} = f(\vec{u})
            := (a_1 , \ldots , a_k, \underbrace{0, \ldots, 0}_{m - ks}) \in \R^m ,
    \end{equation*}
    where $f\colon A^k \to \R^m$ is injective.
    Note that $\vec{v}$ has exactly $k$ non-zero entries,
    and if two codewords $\vec{u},\vec{u}' \in A^k$ satisfy $\Delta(\vec{u},\vec{u}') \geq \frac{k}{2}$ 
    then the corresponding vectors $f(\vec{u}),f(\vec{u}') \in \R^m$ 
    satisfy \mbox{$\|f(\vec{u}) - f(\vec{u}')\|_1 \geq k$}. 
    Indeed, if two entries of $\vec{u}$ and $\vec{u}'$ differ, $a_j \neq a_j'$, 
    then there are two possibilities for the related coordinates $i \in I_j := [js] \setminus [(j-1)s]$ 
    of $\vec{v} = f(\vec{u})$ and $\vec{v}' = f(\vec{u}')$:
    \begin{itemize}
        \item There exists $i^\ast \in I_j$ with $|v_{i^\ast} - v_{i^\ast}'| = 2$,
            and $v_i = v_i' = 0$ for all $i \in I_j \setminus\{i^\ast\}$.
        \item There are two indices $i_1 \neq i_2$ in $I_j$ where $|v_{i_1}| = |v_{i_2}'| = 1$,
            further, $v_i = 0$ for all $i \in I_j \setminus\{i_1\}$, 
            as well as $v_i'=0$ for all $i \in I_j \setminus\{i_2\}$.
    \end{itemize}
    Along the same lines, 
    one can prove $\|f(\vec{u}) - f(\vec{u}')\|_p^p \geq \frac{k}{2} \cdot \min\{2^p, 1^p+1^p\} = k$ 
    for $p \in [1,\infty)$.
    For $p=\infty$, the bound $\|f(\vec{u}) - f(\vec{u}')\|_\infty \geq 1$ 
    follows directly from a single differing entry.
    By Lemma~\ref{lem:Gilbert-Varshamov}, there exists a set $T \subseteq A^k$ 
    with minimal Hamming distance $\frac{k}{2}$ and $\#T > q^{k/10}$.
    The set $S := f(T) \subseteq \R^m$ of the corresponding vectors with entries in $\{-1,0,1\}$ 
    then satisfies the assertion.
\end{proof}

\begin{corollary} \label{cor:U(v)}
    Let $k,m \in \N$ with $m \geq 3k$ and consider the set $W_k^m$ as in~\eqref{eq:W_k}.
    For $\vec{v} \in W_k^m$ and $p \in [1,\infty]$ define the set
    \begin{equation*}
        U(\vec{v}) := \left\{\vec{w} \in W_k^m\colon 
                            \|\vec{w} - \vec{v}\|_p < \tfrac{1}{2}\, k^{1/p}
                    \right\}.
    \end{equation*}
    Then it holds
    \begin{equation*}
        \frac{\#W_k^m}{\#U(\vec{v})} > \left(\frac{m}{k}\right)^{k/10}.
    \end{equation*}
\end{corollary}
\begin{proof}
    All sets $U(\vec{v})$ for $\vec{v} \in W_k^m$ have the same cardinality.
    Let $S \subset W_k^m$ be a set with minimal $\ell_p$-separation $k^{1/p}$ 
    as in Lemma~\ref{lem:SsubsetW}.
    For this we obtain the following disjoint union:
    \begin{equation*}
        \bigsqcup_{\vec{v} \in S} U(\vec{v}) \subseteq W_m^k.
    \end{equation*}
    Hence, $\# S \cdot \# U(\vec{v}) \leq \#W_m^k$ for any $\vec{v} \in W$.
    Rearranging for $\# S$ and using Lemma~\ref{lem:SsubsetW} proves the assertion.
\end{proof}

\begin{remark}
    For $p=\infty$, the corollary above simplifies with $\# U_{\vec{v}} = 1$,
    hence, only an estimate for the cardinality of $W_k^m$ is required.
    Classical bounds give
    \begin{equation*}
        \left(\frac{2m}{k}\right)^k 
            \leq \#W_k^m = \binom{m}{k} 2^k 
            \leq \left(\frac{2em}{k}\right)^k.
    \end{equation*}
\end{remark}

\bibliographystyle{amsplain}

\bibliography{lit}

@book{adler2009random,
  author={R.J.~Adler and J.E.~Taylor},
  title={{R}andom {F}ields and {G}eometry},
  series={Springer Monographs in Mathematics},
  year={2009},
  publisher={Springer New York}
}

@Article{ BIPW10,
        author = "Khanh Do Ba and P.~Indyk and E.~Price and D.P.~Woodruff",
        title = "Lower Bounds for Sparse Recovery",
        journal = "Proceedings of the twenty-fifth annual {ACM-SIAM} symposium on {D}iscrete {A}lgorithms",
        pages = "1190-1197",
        year = "2010"}

@Article{ H92,
        author = "S.~Heinrich",
        title = "Lower bounds for the complexity of {M}onte {C}arlo function approximation ",
        journal = "J.~Complexity", 
        volume = "8",
        issue = "3",
        pages = "277-300",
        year = "1992"
}

@Article{ Ku17,
        author = "R.J.~Kunsch",
        title = "High-Dimensional Function Approximation: Breaking the Curse with {M}onte {C}arlo Methods",
        journal = "Dissertation, FSU Jena, available on arXiv:1704.08213 [math.NA]",
        year = "2017"
}

@Article{ KNW24,
        author = "R.J.~Kunsch and E.~Novak and M.~Wnuk",
        title = "Randomized approximation of summable sequences -- adaptive and non-adaptive",
        journal = "J.~Approximation",
        volume = "304",
        pages = "106056",
        year = "2024"
}

@Article{ KW24b,
    author = "R.J.~Kunsch and M.~Wnuk",
    title = "Uniform approximation of vectors using adaptive
randomized information",
    journal = {J.~Approximation},
    volume = {313},
    pages = {106216},
    year = {2026},
    sortkey = {KW24b},
}

@Article{ KW24c,
    author = "R.J.~Kunsch and M.~Wnuk",
    title = "Adaptive and non-adaptive randomized approximation of high-dimensional vectors",
    journal = {J.~Approximation},
    volume = {318},
    pages = {106318},
    year = {2026},
    sortkey = {KW24c},
}

@book{vL98,
    author = "J.H.~van~Lint",
    title = "Introduction to coding theory",
    publisher = "Springer",
    year = 1998
}

@Article{ Ma91,
    author = "P.~Math\'e", 
    title = "Random approximation of {S}obolev embeddings", 
    journal = "J.~Complexity", 
    volume = "7", 
    pages = "261-281",
    year = "1991"
}

@InProceedings{Pis86,
    author="G.~Pisier",
    title="Probabilistic methods in the geometry of {B}anach spaces",
    booktitle="Probability and Analysis",
    year="1986",
    publisher="Springer Berlin Heidelberg",
    pages="167--241",
}

@inproceedings{PW12,
  title={Lower bounds for adaptive sparse recovery},
  author={E.~Price and D.~Woodruff},
  booktitle={Proceedings of the twenty-fourth annual ACM-SIAM symposium on Discrete algorithms},
  pages={652--663},
  year={2013},
  organization={SIAM}
}

\end{document}